\documentclass[letterpaper, 10 pt, conference]{ieeeconf}  % Comment this line out if you need a4paper

\IEEEoverridecommandlockouts                              % This command is only needed if 
\usepackage[cmex10]{amsmath}

\usepackage{cases}
\usepackage[hidelinks]{hyperref}
\usepackage{amsmath}
\usepackage{amssymb}
\usepackage{mathrsfs}
\usepackage{array}
\usepackage{calc}
\usepackage{graphicx}
\usepackage{color}
\usepackage{geometry}
\newtheorem{proposition}{Proposition}
\newenvironment{prop}{\begin{samepage}\begin{proposition}}{\end{proposition}\end{samepage}}
\newtheorem{remark}{Remark}

\begin{document}
\title{\LARGE \bf
Novel Approach Towards Global Optimality of Optimal Power Flow Using Quadratic Convex Optimization
}

%% To specify the authors when (number of affiliations <= 2)
%\author{
%\IEEEauthorblockN{Hadrien Godard\\ Jean Maeght\\ Manuel Ruiz}
%\IEEEauthorblockA{RTE R\&D \\
%Versailles, France\\
%}
%\and
%\IEEEauthorblockN{Sourour Elloumi\\ Am\'elie Lambert}
%\IEEEauthorblockA{CEDRIC \\
%CNAM \\
%Paris, France\\
%}
%}

\author{Hadrien Godard$^{1,2,3}$, Sourour Elloumi$^{2,3}$, Am\'elie Lambert$^{2}$, Jean Maeght$^{1}$ and Manuel Ruiz$^{1}$
\thanks{$^{1}$ R\&D Division, RTE (French TSO), 9 rue de la Porte de Buc, 78000, Versailles, France}
\thanks{$^{2}$CEDRIC, CNAM, 292 rue Saint-Martin, 75003, Paris, France}
\thanks{$^{3}$UMA, ENSTA, 828 Boulevard des Maréchaux, 91120, Palaiseau, France}
}

%% To specify the authors when (number of affiliations > 2)
% \author{\IEEEauthorblockN{Hadrien Godard\IEEEauthorrefmark{1}\IEEEauthorrefmark{2}\IEEEauthorrefmark{3},
%Sourour Elloumi\IEEEauthorrefmark{2}\IEEEauthorrefmark{3},
%Am\'elie Lambert\IEEEauthorrefmark{2}, 
%Jean Maeght\IEEEauthorrefmark{1} and
%Manuel Ruiz\IEEEauthorrefmark{1}}
% \IEEEauthorblockA{\IEEEauthorrefmark{1} R\&D Division, RTE (French TSO), 9 rue de la Porte de Buc, 78000, Versailles, France}
 %\IEEEauthorblockA{\IEEEauthorrefmark{2} CEDRIC, CNAM, 292 rue Saint-Martin, 75003, Paris, France}
 %\IEEEauthorblockA{\IEEEauthorrefmark{3} UMA, ENSTA, 828 Boulevard des Maréchaux, 91120, Palaiseau, France}
 %}

\maketitle

%---------------------------------------------------------------------------------------------------------------------------------------------------------------------------------------------------------------------------------------------------------------------------------------

\begin{abstract}
Optimal Power Flow (OPF) can be modeled as a non-convex Quadratically Constrained Quadratic Program (QCQP).
Our purpose is to solve OPF to global optimality.
To this end, we specialize the Mixed-Integer Quadratic Convex Reformulation method (\texttt{MIQCR}) to (OPF).

This is a method in two steps.
First, a Semi-Definite Programming (SDP) relaxation of (OPF) is solved.
Then the optimal dual variables of this relaxation are used to reformulate OPF into an equivalent new quadratic program, where all the non-convexity is moved to one additional constraint.
In the second step, this reformulation is solved within a branch-and-bound algorithm, where at each node a \textit{quadratic and convex} relaxation of the reformulated problem, obtained by relaxing the non-convex added constraint, is solved.
The key point of our approach is that the lower bound at the root node of the branch-and-bound tree is equal to the SDP relaxation value.

We test this method on several OPF cases, from two-bus networks to more-than-a-thousand-buses networks from the MATPOWER repository.
Our first results are very encouraging.
\end{abstract}

% Use this to place sponsorships
%\thanksto{Applicable sponsors, if any, should be placed using the \emph{thanksto} command}

%---------------------------------------------------------------------------------------------------------------------------------------------------------------------------------------------------------------------------------------------------------------------------------------

\section{Introduction}

The Optimal Power Flow (OPF) problem deals with determining power production at different nodes of an electric network where a production cost is minimized.
If we model this network with an AC framework, the optimization problem is quadratic and non-convex~\cite{low}.

(OPF) can be solved with general purpose solvers such as Baron~\cite{baron}.
A branch-and-bound algorithm specialized to (OPF) has been introduced in~\cite{gopala} using the SDP rank relaxation~\cite{lavaei2012zero} as a lower bound provider.
More recently~\cite{sun} introduces an SOCP-based branch-and-bound algorithm.

Our goal is to design a branch-and-bound algorithm that closes the gap between lower and upper bounds.
The lower bound is obtained with the rank relaxation and the upper bound is obtained with a feasible point computed, for instance, by an interior point method.
It has been already observed that this gap is quite small for (OPF) instances as these lower and upper bounds are very sharp in general~\cite{arxiv}.

To this end, we work on a specialization of \texttt{MIQCR}~\cite{lambert1,lambert2,lambert3} which is a method designed to solve non-convex and mixed integer quadratic programs to global optimality.
\texttt{MIQCR} works in two steps: first a semi-definite relaxation is used to reformulate the problem into another quadratic program.
Then this reformulated problem is solved within a branch-and-bound framework where at each node a \textit{quadratic and convex} program is solved to get a local lower bound.
A key advantage of this method is that it requires the solution of only one SDP relaxation as a preprocessing step.
Then, the strength of this SDP lower bound is captured onto quadratic and convex programming.
However, the solution of this large SDP is the bottleneck of \texttt{MIQCR} when handling larger instances.
The contribution of this paper is a specialization of MIQCR to (OPF) where we prove that we can reach the strengthened lower bound of the MIQCR method by solving a smaller semidefinite relaxation than in the original approach.

\bigskip

In this paper, we adapt method \texttt{MIQCR} to the OPF problem.
In Section~2, we recall the formulation of the OPF problem as a quadratic program.
In Section~3 we introduce the semi-definite relaxation used in our algorithm, known as the rank (or Shor) relaxation of the OPF problem~\cite{lavaei2012zero}.
We also give in this section a new proof of the strong duality of the rank relaxation.
Then, in Section~4 and in Section~5, we present our contribution: the specialization of \texttt{MIQCR} to solve the OPF problem.
In Section~4, we prove that solving a smaller SDP relaxation than in the original \texttt{MIQCR} approach is sufficient to reformulate the OPF problem and to attain the sharp rank relaxation bound at the root node of the branch-and-bound tree.
Indeed, solving the dual of the rank relaxation is equivalent to finding the best quadratic reformulation of an OPF problem among the \texttt{MIQCR} family of reformulations.
Moreover, in Section~5, we present a branch-and-bound framework to solve the reformulated problem.
Finally, in Section~6, we illustrate our method by computational experiments on small and medium-sized instances of OPF problems.
Section~7 draws a conclusion.

\section*{Notations}

\begin{itemize}
\item $\textbf{i}$ is the complex number whose real part is null and imaginary one equals one.
\item $\mathbb{S}_{m}(\mathbb{R})$ is the set of symmetric matrices on $\mathbb{R}^{m}$.
\item $0_{m}$ is the zero matrix of size $m$.
\item $Id_m$ is the identity matrix of size $m$.
\item $\lambda_{min}(M)$ is the smallest eigenvalue of a symmetric matrix $M$.
\item $\texttt{v}(P)$ is the optimal value of an optimization problem $(P)$.
\item $<A,B>$ is the canonic scalar product between matrices $A$ and $B$.
\item $|E|$ is the cardinal of a set $E$.
\end{itemize}
%---------------------------------------------------------------------------------------------------------------------------------------------------------------------------------------------------------------------------------------------------------------------------------------

\section{A quadratic formulation of the OPF problem}

Driving power flows from producers to consumers in an electric network constitutes the OPF problem.
Usually the amount of consumed electric power is known at each node of the network.
On the contrary the production is unknown and OPF deals with determining its value.
The goal is to minimize electric power production costs under the constraint that the demand is satisfied at each node, and that active and reactive powers are box constrained for each production unit.
Because real electric transmission networks work with alternative current, one must consider voltage at each node in order to compute where the power flows through the network.
Engineering limits such as bounds on voltage magnitude are also considered as constraints.
See the bus injection model in~\cite{low} for more precisions on this topic.

We only consider a cost linearly linked to the power produced at each unit.
Note that a linear cost on the power production handles the case where one tries to minimize losses on the network.
The following program is a classic model for the (OPF) problem:

\begin{numcases}{} 
\displaystyle{\min_{V \in \mathbb{C}^{n}, P^G \in \mathbb{R}^{n^g}, Q^G \in \mathbb{R}^{n^g}}} c'P^G \label{obj}\\
\mbox{s.t.} \nonumber \\
\quad P^G_i + \textbf{i} Q^G_i = V^* Y_i V, \quad i \in N^G, \label{Gbilan}\\
\quad -P^D_i - \textbf{i} Q^D_i = V^* Y_i V, \quad i \in N^D, \label{Dbilan}\\
\quad P^{min}_i \leq P^G_i \leq P^{max}_i, \quad i \in N^G, \label{Pbounds}\\
\quad Q^{min}_i \leq Q^G_i \leq Q^{max}_i, \quad i \in N^G, \label{Qbounds} \\
\quad (V^{min}_i)^2 \leq |V_i|^2 \leq (V^{max}_i)^2, \quad i \in N. \label{Vbounds}
\end{numcases}

Where $N$ is the set of network nodes and $n$ their number.
$N^G$, respectively $N^D$, is the set of production, respectively consumption, nodes and $n^g = |N^G|$.   
Variables $V$ are the complex voltages of network buses.
Variables $P^G$ are the generated active powers at production buses.
Variables $Q^G$ are the generated reactive powers at production buses.
$c$ is the vector of linear costs, where $c_i$ is the production cost at node $i \in N^G$.
$Y_i$ is the complex admittance matrix at node $i \in N$.
$P^{min}$ and $P^{max}$ are lower and upper bounds on active generated powers.
$Q^{min}$ and $Q^{max}$ are lower and upper bounds on reactive generated powers. 
$V^{min}$ and $V^{max}$ are lower and upper bounds on voltage magnitudes.

The objective function (\ref{obj}) is linear relatively to active produced powers $P^G$.
Constraints (\ref{Gbilan}) (resp. (\ref{Dbilan})) are power balance equations at production (resp. consumption) nodes.
Constraints (\ref{Pbounds}), (resp. (\ref{Qbounds})), are bounds on active (resp. reactive), produced powers.
Constraints (\ref{Vbounds}) are bounds on voltage magnitudes.

Substituting voltages to other variables, from equations (\ref{Gbilan}), one can see that (OPF) can be modeled as a pure quadratic program, without linear terms, in the real and imaginary parts of voltages.

To simplify the notations, we rewrite the above problem in a more abstract style as the following quadratically constrained quadratic program with $2n$ variables $(OPF)$:

\begin{numcases}{(OPF)} 
\displaystyle{\min_{x \in \mathbb{R}^{2n}}} x^t C x \nonumber\\
\mbox{s.t.} \nonumber \\
\quad  x^t A_k x  \leq  b_k, & $k \in K.$ \nonumber
\end{numcases}

Variables $x$ are the real and imaginary parts of the voltage at each network node.
$C \in \mathbb{S}_{2n}(\mathbb{R})$.
$K$ is the set of constraints indices. At each node there are two constraints that bound the voltage magnitude, and four other constraints that model the complex power balance.
Hence, the number of constraints is $6n$.
$\forall k \in K,  A_k \in \mathbb{S}_{2n}(\mathbb{R})$ and $b_k \in \mathbb{R}$.

The objective function and the constraints are convex if and only if matrices $C$ and $A_k$ are positive semi-definite (PSD) which is not the case for $(OPF)$.

In this formulation with continuous variables there are no natural lower and upper bounds on variables $x$.
However to perform a spatial branch-and-bound algorithm, initial lower and upper bounds ($\ell$ and $u$) on each variable are needed.

To obtain such bounds one can use the fact that the modulus of each complex voltage is upper-bounded. 
Those upper-bounds are also at the heart of the proofs of Propositions~1~and~2.

Suppose that $x_i$ is the real part of the complex voltage at node $i$, and $x_{i+n}$ its imaginary part.
By Constraints~(\ref{Vbounds}), we have:
 \begin{equation}\label{ball_constraints}x_i^2 + x_{i+n}^2 \leq (V_i^{max})^2.\end{equation}
\noindent It follows that:
$$ -V_i^{max} \leq x_i \leq V_i^{max} \textrm{ and } -V_i^{max} \leq x_{i+n} \leq V_i^{max}.$$
Below, we take:
\begin{align}
\label{l_def}\ell_i = \ell_{i+n} =  -V_i^{max},\\
\label{u_def}u_i = u_{i+n} =  V_i^{max},\\
\ell \leq x \leq u.\nonumber
\end{align}

%---------------------------------------------------------------------------------------------------------------------------------------------------------------------------------------------------------------------------------------------------------------------------------------

\section{The rank relaxation of OPF}

In this section we recall the rank relaxation of $(OPF)$, that we call $(SDP)$:
\begin{numcases}{(SDP)} 
\displaystyle{\min_{X \in \mathbb{S}_{2n}(\mathbb{R})}}  <C,X> \nonumber\\
\mbox{s.t.} \nonumber \\
\quad   <A_k,X>  \leq  b_k,  & $k \in K$ \nonumber     \\
\quad  X \succeq  0.  \nonumber 
\end{numcases}

The dual of $(SDP)$ is:
\begin{numcases}{(DSDP)} 
\displaystyle \max_{\alpha \in \mathbb{R}^{|K|}}  \displaystyle \sum_{k \in K} - b_k \alpha_k \nonumber\\
\mbox{s.t.} \nonumber \\
\quad   C + \displaystyle \sum_{k \in K} \alpha_k A_k  \succeq  0, \nonumber     \\
\quad  \alpha_k  \geq  0, & $ k \in K. $ \nonumber 
\end{numcases}
\noindent where $\alpha_k$ is the dual variable associated with constraint $k$.

Strong duality for feasible OPFs has been already proved (see for instance~\cite{josz_these}). In this paper, we propose another proof based on the fact that the modulus of each complex voltage is bounded.
Parts of this proof will be used later to demonstrate Proposition 2.

\begin{prop}
If $(OPF)$ is feasible, there is no duality gap between $(SDP)$ and $(DSDP)$. In other words, strong duality holds for the rank relaxation of feasible OPF problems.
\end{prop}

\begin{proof}
We first prove that if $(OPF)$ is feasible then $(SDP)$ is feasible too.
Indeed from each feasible solution $\tilde{x}$ to $(OPF)$ one can build a feasible solution $\tilde{X} = \tilde{x} \tilde{x}^t$ to $(SDP)$.

Let us now prove that $(DSDP)$ is strictly feasible, i.e. finding $\tilde{\alpha} > 0$ satisfying $C+ \sum_{k\in K} \tilde{\alpha}_k A_k \succ 0$.

In the following we assume that the elements of $K$ are integers from 1 to $|K|$, and that the $n$ first elements of $K$ are the indices of voltage magnitude upper-bound constraints~(\ref{ball_constraints}) on the $n$ network nodes.
For $k$ from $1$ to $n$: all entries of $A_k$ are zeros except the $k$-th and $(k+n)$-th entries of the diagonal, which are equal to 1.
It follows that $\sum_{k=1,...,n} A_k = Id_{2n}$.

For $k > n$ : take $\tilde{\alpha}_k = 1 > 0$.
Consider the matrix \\ $C' = C+ \sum_{k>n} \tilde{\alpha}_k A_k$ .

For $k$ from $1$ to $n$ : \\
take $\tilde{\alpha}_k = \mu \geq 1 + \texttt{max}(0,-\lambda_{min}(C')) > 0$.

Then:
\begin{multline*}
C+ \sum_{k\in K} \tilde{\alpha}_k A_k = C + \sum_{k = 1,..,n} \tilde{\alpha}_k A_k + \sum_{k > n} \tilde{\alpha}_k A_k \\= C' + \mu Id_{2n} \succ 0.
\end{multline*}

\noindent As $\tilde{\alpha}$ has positive entries and $C+ \sum_{k\in K} \tilde{\alpha}_k A_k \succ 0$, $(DSDP)$ is strictly feasible.

To sum up, $(SDP)$ is feasible, $(DSDP)$ is strictly feasible, so strong duality holds.
\end{proof}

We have shown that strong duality holds for the rank relaxation of $(OPF)$, using the fact that the modulus of each complex voltage is bounded.
This is a special case of ball constraints on an optimization problem with complex variables, see~\cite{josz_these} to get more details on this subject.

$(SDP)$ relaxation of $(OPF)$ is known to give a sharp lower bound.
We want to point out that we study transmission networks which are not tree networks and are highly meshed.
Moreover we do not have a sufficient number of phase-shifter-transformers on network to use results from~\cite{net_top}.
Thus rank relaxation does not necessarily lead to an optimal solution.

In the next sections, we present a branch-and-bound algorithm that starts with this sharp lower bound which is based at each node on a \textit{quadratic and convex} relaxation of $(OPF)$.

%---------------------------------------------------------------------------------------------------------------------------------------------------------------------------------------------------------------------------------------------------------------------------------------

\section{An equivalent formulation to $(OPF)$}

The first step of \texttt{MIQCR} consists in reformulating a QCQP like $(OPF)$ into an \textit{equivalent} quadratic problem that has a \textit{quadratic and convex} objective function, linear constraints, and additional variables $Y$ that are meant to satisfy quadratic constraint $Y = xx^t$.

Let $S \in \mathbb{S}_{2n}(\mathbb{R})^+$ be a positive semi-definite matrix. We reformulate $(OPF)$ as:

\begin{small}
\begin{numcases}{(OPF_{S})} 
\displaystyle \displaystyle \min_{x \in \mathbb{R}^{2n}, Y \in \mathbb{S}_{2n}(\mathbb{R})} x^tSx + <C-S,Y>  \nonumber\\
\mbox{s.t.} \nonumber \\
\quad  <A_k,Y>  \leq b_k, \quad k \in K \nonumber     \\
\label{bilinear}\quad  Y=xx^t, \\
\quad \ell \leq x \leq u. \nonumber
\end{numcases}
\end{small}

Observe that the reformulated objective and constraints functions have the same value as the original ones for a same $x$ if Constraints (\ref{bilinear}) are satisfied.
Moreover, the new objective function is convex since matrix $S$ is positive semi-definite. The constraints are now linear, and thus convex.
This is why, this reformulation is called a \textit{quadratic and convex} reformulation.

In $(OPF_{S})$ only Constraint (\ref{bilinear}) is non convex. In a way, all the non-convexity has been moved into this constraint.

To solve $(OPF_{S})$ to global optimality \texttt{MIQCR} uses a branch-and-bound algorithm where, at each node, we relax Constraint (\ref{bilinear}) and add the linear McCormick inequalities (\ref{mc1})-(\ref{mc4})~\cite{Cor76} to tighten the relaxation.

Therefore, at the first node of the branch-and-bound, the \textit{quadratic and convex} relaxation $(\overline{OPF}_{S})$ is solved.

\begin{footnotesize}
\begin{numcases}{(\overline{OPF}_{S})} 
\displaystyle \displaystyle \min_{x \in \mathbb{R}^{2n}, Y \in \mathbb{S}_{2n}(\mathbb{R})} x^tSx + <C-S,Y>  \nonumber\\
\mbox{s.t.} \nonumber \\
\quad  <A_k,Y>  \leq  b_k, \quad k \in K \nonumber     \\
\label{mc1}\quad Y_{ij} \leq u_jx_i + \ell_ix_j -\ell_iu_j, (i,j) \in E \\
\label{mc2}\quad Y_{ij} \leq \ell_jx_i + u_ix_j -u_i\ell_j, (i,j) \in E \\
\label{mc3}\quad Y_{ij} \geq u_jx_i + u_ix_j -u_iu_j, (i,j) \in E\\
\label{mc4}\quad Y_{ij} \geq \ell_jx_i + \ell_ix_j -\ell_i\ell_j, (i,j) \in E.
\end{numcases}
\end{footnotesize}

\noindent where $\ell,u$ are defined as in~(\ref{l_def}) and~(\ref{u_def}) and \\$E = \Big \{ (i,j) \in \{1,\ldots, 2n\}^2 \, : \, i \leq j \Big \}$.

Notice that in practice, the full variables matrix $Y$ is not considered.
Coefficient $Y_{ij}$ is considered if and only if the entry $(i,j)$ is non zero in a matrix among $\{C,A_1,\ldots,A_{|K|}\}$. 

Every positive semi-definite matrix gives a different reformulation. In the particular case where $S=0_{2n}$, the objective function is linear.
It is the linearization of $(OPF)$.
For example the Baron solver~\cite{baron} relies on this linearization within a branch-and-bound framework.

We are now interested in finding a "best" matrix $S$, i.e. a matrix $S$ which gives the largest lower bound at the root node of the branch-and-bound tree. That is to say, a matrix $S$ which maximizes the value of $(\overline{OPF}_{S})$.
More formally we want to solve:
$$ \displaystyle \max_{S\succeq 0} \texttt{v}(\overline{OPF}_S).$$
It is proved in~\cite{lambert2} that a best matrix $S^*$ can be computed from optimal dual variables of a semi-definite relaxation of $(OPF)$ which is $(SDP)$ but with additional constraints and variables to raise the McCormick's inequalities.

In this paper, we characterize $(SDP)$, which does not contain the McCormick inequalities, as the relaxation that can be used to compute a best positive semi-definite matrix.
The fact that the McCormick inequalities are redundant in the $(SDP)$ relaxation is a significant result and is the main difference between Proposition 2 and the result in~\cite{lambert2}.
As a consequence, to reformulate $(OPF)$, we have to solve a problem with a smaller size than the one solved in the original method \texttt{MIQCR}.
Moreover, we prove that the optimal value of this "best" \textit{quadratic and convex} relaxation is equal to the optimal value of $(SDP)$.

\begin{prop}
Let $\alpha^*$ be an optimal solution to $(DSDP)$, take:
	$$S^* =  C + \displaystyle \sum_{k \in K} \alpha_k^* A_k,$$

\noindent If $\ell$ and $u$ are defined by (\ref{l_def}) and (\ref{u_def}) we have:
$$\texttt{v}(SDP) = \texttt{v}(\displaystyle \max_{S\succeq 0} \texttt{v}(\overline{OPF}_S)) = \texttt{v}(\overline{OPF}_{S^*}).$$
\end{prop}

\begin{proof}
For each PSD matrix $S$, let us introduce the optimization problem:

\begin{small}
\begin{numcases}{(LR_{S})} 
\displaystyle \displaystyle \min_{x \in \mathbb{R}^{2n}, Y \in \mathbb{S}_{2n}(\mathbb{R})} x^tSx + <C-S,Y>  \nonumber\\
\mbox{s.t.} \nonumber \\
\quad  <A_k,Y>  \leq b_k, \quad k \in K \nonumber
\end{numcases}
\end{small}

$(LR_S)$ is the relaxation of $(\overline{OPF}_S)$ where inequalities (\ref{mc1})-(\ref{mc4}) have been dropped. Thus, it is a relaxation of $(OPF_S)$.

The proof is divided in two parts.

\begin{itemize}
\item First we prove that:
$$\texttt{v}(DSDP) = \texttt{v}(\displaystyle \max_{S\succeq 0} \texttt{v}(LR_S)) = \texttt{v}(LR_{S^*}).$$

Let us rewrite the dual of the SDP relaxation by introducing a slack matrix $S$:
\begin{numcases}{(DSDP)} 
\displaystyle \max_{S \succeq 0, \alpha \geq 0}  \displaystyle \sum_{k \in K} - b_k \alpha_k \nonumber\\
\mbox{s.t.} \nonumber \\
\quad   S = C + \displaystyle \sum_{k \in K} \alpha_k A_k. \nonumber
\end{numcases}

For a given $S \succeq 0$, the optimization problem in $\alpha$ is a linear program.
We replace it by its LP-dual and obtain the equivalent problem: 
\begin{numcases}
\displaystyle \max_{S \succeq 0} \displaystyle \min_{Y}  <C-S,Y> \nonumber\\
\mbox{s.t.} \nonumber \\
\quad  <A_k,Y>  \leq b_k, \quad k \in K. \nonumber
\end{numcases}

Now we can observe that as S is positive semi-definite, one can add $x^tSx$ to the objective function together with variables $x$.
Indeed $x$ will be equal to 0 in any optimal solution.
Therefore:
\begin{numcases}{\texttt{v}(DSDP) = \texttt{v}} 
\displaystyle \max_{S \succeq 0} \displaystyle \min_{x,Y}  x^tSx + <C-S,Y> \nonumber\\
\mbox{s.t.} \nonumber \\
\quad  <A_k,Y>  \leq b_k, \quad k \in K. \nonumber
\end{numcases}

We have proved:
$$\texttt{v}(DSDP) = \texttt{v}(\displaystyle \max_{S\succeq 0} \texttt{v}(LR_S)) = \texttt{v}(LR_{S^*}).$$

\item Now, in the second part, we prove that:
\begin{equation*}
\forall S \succeq 0, \texttt{v}(\overline{OPF}_{S}) \leq \texttt{v}(SDP) \leq \texttt{v}(\overline{OPF}_{S^*}).
\end{equation*}

From the first part, and by Proposition~1, it follows that $\texttt{v}(SDP) = \texttt{v}(LR_{S^*})$,
and, as $(LR_{S^*})$ is a relaxation of $\overline{OPF}_{S^*}$:
$$\texttt{v}(SDP) = \texttt{v}(LR_{S^*}) \leq \texttt{v}(\overline{OPF}_{S^*}).$$

Let us now prove that:
$$\forall S \succeq 0, \texttt{v}(\overline{OPF}_{S}) \leq \texttt{v}(SDP).$$

Let $\bar{X}$ be a solution to $(SDP)$.
Let us show that $(\bar{x} = 0,\bar{Y} = \bar{X})$ is a feasible solution to $(\overline{OPF}_S)$ with a lower objective value.

The inequalities $<A_k,\bar{Y}> \leq b_k, k \in K$ are trivially satisfied.
Let us now prove that (\ref{mc1})-(\ref{mc4}) are satisfied.
Which amounts to prove:
\begin{align}
\label{p1}\bar{X}_{i,j} \leq V_i^{max}V_j^{max}, (i,j) \in E, \\
\label{p2}\bar{X}_{i,j} \geq -V_i^{max}V_j^{max}, (i,j) \in E.
\end{align}

When $i=j$, as $\bar{X}\succeq 0$, then for all $i \in \{1, \ldots, 2n\}, \bar{X}_{i,i} \geq 0$.
Moreover, from (\ref{ball_constraints}):
$$\forall i \in \{1,\ldots,n\}: \bar{X}_{i,i} + \bar{X}_{i+n,i+n} \leq (V_i^{max})^2.$$

\noindent And therefore:
\begin{small}
\begin{equation}\label{ii}\bar{X}_{i,i} \leq (V_i^{max})^2 \mbox{ and } \bar{X}_{i+n,i+n} \leq (V_i^{max})^2.\end{equation}
\end{small}
When $(i \neq j)$, as $\bar{X} \succeq 0, \bar{X}_{i,i}\bar{X}_{j,j} - \bar{X}_{i,j}^2 \geq 0$.
From (\ref{ii}), it follows that $\bar{X}_{i,j}^2 \leq (V_i^{max})^2(V_j^{max})^2$.
Then (\ref{p1}) and (\ref{p2}) are satisfied.

Let us now compare the objective solution values of $\bar{X}$ and $(\bar{x},\bar{Y})$:
$$<C,\bar{X}> - \bar{x}^tS\bar{x} - <C-S,\bar{Y}> = <S,\bar{X}> \geq 0$$ as $S$ and $\bar{X}$ are PSD.

Therefore $(\bar{x},\bar{Y})$ has a lower solution value than $\bar{X}$.

To sum up for any PSD matrix $S$, $(\overline{OPF}_S)$ has a lower optimal solution value than $(SDP)$.
\end{itemize}

We can now conclude that $\texttt{v}(\overline{OPF}_{S^*}) = \texttt{v}(SDP)$ and that $S^*$ maximizes the value of $(\overline{OPF}_{S})$.
\end{proof}

\begin{remark}
In the proof above, we demonstrate that any solution of $(SDP)$ "satisfies" the McCormick inequalities.
This is because bounds $\ell$ and $u$ were obtained with constraints (\ref{ball_constraints}) which are in $(SDP)$. See Section II.
\end{remark}

\begin{remark}
The optimal matrix $S^*$ is not unique, other matrices may give a root node relaxation with the same value.
\end{remark}

Proposition~2 ensures that the lower bound obtained at the root node of our branch-and-bound framework is equal to the rank relaxation bound.
For many test cases, this bound seems to be very sharp~\cite{arxiv}.
To solve $(SDP)$ and compute $S^*$ matrix, one can use the solver introduced in~\cite{arxiv,molzahn}.

Those results allow us to build the following algorithm to solve the OPF problem to global optimality:

\noindent\fbox{\parbox{\linewidth-2\fboxrule-2\fboxsep}{
\begin{enumerate}
\item Solve the rank relaxation and deduce optimal dual variables $\alpha^*$.
\item Define the PSD matrix $S^* = C + \sum_{k \in K} \alpha^*_k A_k$.
\item Solve $(OPF_{S^*})$ within a branch-and-bound algorithm.
\end{enumerate}
}}

%---------------------------------------------------------------------------------------------------------------------------------------------------------------------------------------------------------------------------------------------------------------------------------------

\section{Solution within a branch-and-bound algorithm}

In the previous section we showed how to build an "optimal" reformulation of $(OPF)$ in the sense that it maximizes the lower bound at the root node of our branch-and-bound tree.
In this section we describe the second step of the \texttt{MIQCR} method: the solution within a branch-and-bound algorithm.
This algorithm is used to solve $(OPF_{S^*})$ and hence $(OPF)$.

Let us recall that a branch-and-bound is an enumeration tree used to solve an optimization problem.
Each node of the tree represents a sub problem of the original one.
There are multiple ways to divide the original problem into sub problems.
The classic way is to divide each variable interval into different subintervals, those branch-and-bound algorithms are called spatial branch-and-bound.
We choose to implement this type of branch-and-bound, that is why we need bounds $\ell$ and $u$ on variables $x$.
To sum up, at each node we modify values of $\ell$ and $u$ to build the sub problem.

A branch-and-bound implementation is defined by:

\begin{itemize}
\item Actions performed at each node of the tree,
\item Next node selection strategy.
\end{itemize}

\subsection{How to deal with a node ?}

At each node we solve $(\overline{OPF}_{S_*})$ where bounds $\ell$ and $u$ are different.
This change modifies the relaxation value since $\ell$ and $u$ are involved in the McCormick inequalities (\ref{mc1})-(\ref{mc4}).
We recall that this relaxation is \textit{convex and quadratic} and that it gives a lower bound of the node subproblem.

Next step depends on the result of the node relaxation:

\begin{itemize}
\item If the relaxation is infeasible: the branch is pruned.
\item If the lower bound from relaxation is greater than the best current upper bound: the branch is  pruned.
\item If the solution $(\bar{x},\bar{Y})$ from the relaxation satisfies constraint (\ref{bilinear}): $\bar{x}$ is then a solution of $(OPF)$, the branch is pruned and the upper bound is potentially updated.
\item Else: two new nodes are built as children of the current node.
\end{itemize} 

\noindent \textbf{About branching}: To build two children nodes from a parent node, a branching variable $(x_b)$ and a branching value $(x_b^s)$ are chosen.

\textit{Variable selection strategy}:
$x_b$ is chosen among variables $(i_b,j_b)$ that violate the most Equality (\ref{bilinear}) (for the Euclidian norm).

\textit{Interval division}:
$x_b^s$ is chosen between the middle of the current interval of variable $x_b$ (denoted $x_b^m$) and $\bar{x}_b$, the value of $x_b$ in the node relaxation solution.
Let $\delta$ be a parameter between 0 and 1:
$$x_b^s = \delta \bar{x}_b + (1-\alpha) x_b^m.$$

Lower bounds on the local optimal solution value are computed at each node, however we need to find upper bounds and we cannot rely on relaxations results to do so.
That is why every three nodes, a heuristic, consisting in finding a point satisfying first order optimality conditions of $(OPF)$, is launched.

\subsection{How to select the following node ?}

At each node is attached a \textit{potential} which is the lower bound found at its parent node.
At this node the relaxation value cannot be lower than this \textit{potential}.

Our next node selection strategy is designed as a "best-first" strategy.
Indeed we want to see the global lower bound of the tree (which is the minimum among the \textit{potentials} of the leafs in the tree) increase.
To do so, we select the node with the lowest \textit{potential} as the following node to handle.

Branch-and-bound algorithm is terminated when the relative difference between the lowest upper bound (the best feasible solution) and the global lower bound is less than an $\epsilon$-value.

%---------------------------------------------------------------------------------------------------------------------------------------------------------------------------------------------------------------------------------------------------------------------------------------

\section{Numerical experiments}

To illustrate our method, we present results on a batch of 20 OPF instances involving networks from 2 to 1354 nodes.
All these instances come from the MATPOWER repository~\cite{matpower}, except the "WB" and "LMBM" instances that come from~\cite{bukhsh}.

In our experiments (as in~\cite{arxiv}) we do not consider any constraints on current magnitudes, and we only consider the linear part on the active power cost.

For each instance, we launch our implementation of \texttt{MIQCR} along with the Baron solver~\cite{baron}, keeping default options.

For \texttt{MIQCR}, the rank relaxation is solved with the Mosek solver~\cite{mosek} and a chordal decomposition to exploit the problem sparsity. 
The \textit{quadratic and convex} relaxation at each branch-and-bound node is solved with the Xpress~\cite{xpress} solver.
Heuristic performed at each three nodes is based on the solution of the first order optimality conditions of $(OPF)$ with the Knitro solver~\cite{knitro,knitro2}.
To perform computation we use the following parameters:
\begin{center}
\begin{tabular}{|l|l|}
\hline
Absolute feasible tolerance for the heuristic & $10^{-5}$\\ \hline
Final $\epsilon$ gap & $10^{-5}$\\ \hline
$\delta$ parameter to perform branching & 0.5 \\ \hline
\end{tabular}
\end{center}

Hence a 0\% gap means that the gap is under the $\epsilon$ parameter value.
The choice of $\delta = 0.5$ came after some numerical experiments along with the reading of~\cite{belotti2009branching}.

\begin{small}
\begin{table}
	\begin{center}
	\begin{tabular}{|c|c|c|c|c|}
	\hline
	& \multicolumn{2}{|c|}{MIQCR} & \multicolumn{2}{|c|}{BARON} \\ \hline
	\textit{Name} & \textit{Root gap} & \textit{Gap-Time} & \textit{Root gap} & \textit{Gap-Time} \\ \hline
	\textit{WB2} & 2.2\% & 1s & 2.2\% & 1s \\ \hline
	\textit{LMBM3} & 0\% & 1s & 0\% & 1s \\ \hline
	\textit{WB3} & 0\% & 1s & 0\% & 1s \\ \hline
	\textit{WB5} & 16.7\% & 23s & 16.8\% & 1.92s \\ \hline
	\textit{6ww} & 0\% & 1s & 0.2\% & 1s \\ \hline
	\textit{9} & 0\% & 1s & 0\% & 1.34s \\ \hline
	\textit{9mod} & 0.1\% & (0.1\%)  & 12.4\% & 5.83s \\ \hline
	\textit{14} & 0\% & 1s & 100\% & (21.1\%) \\ \hline
	\textit{22loop} & 0\% & 1s & 31.6\% & (31.6\%) \\ \hline
	\textit{30} & 0\% & 1s & 100\% & (100\%) \\ \hline
	\textit{39} & 0\% & 2s & 100\% & (100\%) \\ \hline
	\textit{39mod1} & 0.1\% & (0.1\%) & 100\% & (100\%) \\ \hline
	\textit{39mod2} & 0.1\% & (0.1\%) & 100\% & (100\%) \\ \hline
	\textit{57} & 0\% & 1s & 100\% & (100\%) \\ \hline
	\textit{89pegase} & 0\% & 2s & 72\% & (72\%) \\ \hline
	\textit{118} & 0\% & 3s & 100\% & (100\%) \\ \hline
	\textit{118mod} & 0\% & 7s & 100\% & (100\%) \\ \hline
	\textit{300} & 0\% & 10s & 100\% & (100\%) \\ \hline
	\textit{300mod} & 0.1\% & (0.1\%) & 100\% & (100\%) \\ \hline
	\textit{1354pegase} & 0\% & 204s & 69\% & (69\%) \\ \hline
	\end{tabular}
	\end{center}
	\caption{Results on 20 OPF test cases, with a time limit of 300s}\label{table1}
\end{table}
\end{small}
Table~1 presents numerical results where each line refers to one OPF test case.
In the first column, each instance name contains the number of nodes in the associated electric network. We recall that in $(OPF)$ the number of real variables is twice the number of nodes.
For each solver, the \textit{Root gap} column gives the relative gap between the best known solution for the instance and the lower bound found at the root node of the branch-and-bound tree, i.e., if the best known solution is denoted by $UB$ and the lower bound by $LB$, the gap equals $\frac{UB-LB}{UB}$.
The \textit{Gap-Time} column gives the execution time if global optimality is reached within five minutes of computation, if this is not the case, this column gives the relative final gap between the best known solution and the final lower bound.

On small test cases (under 10 nodes) the gap is closed within the branch-and-bound tree by both solvers, except for \textit{9mod} instance where \texttt{MIQCR} fails.
We observe that the root gap of \texttt{MIQCR} is very tight.
This reflects the quality of the SDP rank relaxation lower bound as already observed by several authors~\cite{arxiv,lavaei2012zero}.
On the contrary, the root gap obtained with the Baron solver~\cite{baron} is often very large for instances with more than 10 nodes.
Moreover its root gap is not improved during the five minutes of branch-and-bound computation.

Notice that a feasible solution is always found for each test case by \texttt{MIQCR}, and may be improved within the branch-and-bound.
For instance, for \textit{case300} first upper bound found equals 475783 and it is improved to 475462.2 during branch-and-bound tree iterations.

In~\cite{josz_sm_ca}, authors successfully use the Lasserre hierarchy~\cite{lasserre2009convexity} to also solve smallest instances of (OPF) problems to global optimality.
Although their method can be extended to deal with some well-conditioned larger instances~\cite{josz_these, josz}, it fails with largest generic ones.

\section{Conclusion}

In this paper we show how to adapt the \texttt{MIQCR} method to the OPF problem.
We prove that the well-known rank relaxation is sufficient to build an "optimal" quadratic reformulation of $(OPF)$.
This result can be extended to each quadratic problem with complex variables whose magnitudes are upper-bounded.

When solving this optimal reformulation within a branch-and-bound framework, we can close the gap between rank relaxation and known feasible solutions.

First numerical results are encouraging.
Future work consists in a specialization of the branch-and-bound framework in order to close the gap by raising the lower bound in larger instances of the OPF problem.
To do so we will focus on bound tightening techniques.

%---------------------------------------------------------------------------------------------------------------------------------------------------------------------------------------------------------------------------------------------------------------------------------------

\end{document}